\documentclass{article}
\usepackage[utf8]{inputenc}

\def\q#1.{{\bf #1. }}
\usepackage{amsmath,amsthm,amssymb}
\def\x{{\bf x}}
\def\y{{\bf y}}
\def\d{{\bf d}}
\def\a{{\bf a}}
\def\b{{\bf b}}
\def\p{{\bf p}}
\def\q{{\bf q}}
\DeclareMathOperator{\tr}{tr}
\DeclareMathOperator{\ch}{ch}
\DeclareMathOperator{\col}{col}
\DeclareMathOperator{\AT}{AT}

\let\leqslant\leqslant

\let\geqslant\geqslant
\usepackage{graphicx}%Вставка картинок правильная

\usepackage{float}%"Плавающие" картинки

\usepackage{wrapfig}%Обтекание фигур (таблиц, картинок и прочего)

\usepackage{color}

\usepackage{url}

\newtheorem{theorem}{Theorem}
\newtheorem{proposition}[theorem]{Proposition}
\newtheorem{corollary}[theorem]{Corollary}
\newtheorem{remark}{Remark}

\begin{document}

\title{Alon -- Tarsi numbers of direct products}

\author{
Fedor Petrov\thanks{St. Petersburg State University, St. Petersburg, Russia
{\sl f.v.petrov@spbu.ru}.}
\and
Alexey Gordeev\thanks{The Euler International Mathematical Institute, St. Petersburg, Russia
{\sl gordalserg@gmail.com }.}}
% \date{}

\maketitle

\begin{abstract}
We provide a general framework on
the coefficients of the graph polynomials
of graphs which are Cartesian products.
As a corollary, we prove that
if $G=(V,E)$ is a graph with degrees
of vertices $2d(v), v\in V$,
and the graph polynomial
$\prod_{(i,j)\in E} (x_j-x_i)$
contains an ``almost central'' monomial
(that means a monomial $\prod_v x_v^{c_v}$, where
$|c_v-d(v)|\leqslant 1$ for all $v\in V$), then
the Cartesian product 
$G\square C_{2n}$ is $(d(\cdot)+2)$-choosable.
\end{abstract}

\section{Introduction}

Let $\mathbb{F}$ be a field, $\x=(x_1,\ldots,x_n)$ a set of 
variables. For $A\subset \mathbb{F}$
and $a\in A$ denote 
$$
D(A,a):=\prod_{b\in A\setminus a}
(a-b).
$$
For a multi-index $\d=(d_1,\ldots,d_n)\in \mathbb{Z}_{\geqslant 0}^n$
denote $|\d|=d_1+\ldots +d_n$,
$\x^\d=\prod_{i=1}^n x_i^{d_i}$.
For a polynomial $f\in \mathbb{F}[\x]$
denote by $[\x^\d]f$ the coefficient of
monomial $\x^\d$ in polynomial $f$.

\begin{theorem}[Combinatorial Nullstellensatz \cite{Alon1999comb}]
Choose 
arbitrary subsets 
$A_i\subset \mathbb{F}$,
$|A_i|=d_i+1$ for $i=1,\ldots,n$.
Denote 
$A=A_1\times A_2\times\ldots \times A_n$.
For any polynomial $f\in \mathbb{F}[\x]$
such that $\deg f\leqslant |\d|$, if $[\x^{\d}]f\ne 0$,
then there exists $\a\in A$
for which $f(\a)\ne 0$.
\end{theorem}

Alon and Tarsi \cite{Alon1992} suggested to use Combinatorial Nullstellensatz
for list graph colorings. Namely, if $G=(V,E)$ is a 
non-directed graph
with the vertex set $V=\{v_1,\ldots,v_n\}$ and
the edge set $E$, we define its graph polynomial
in $n$ variables $x_1,\ldots,x_n$ as
$$
F_G(\x)=\prod_{(i,j)\in E} (x_j-x_i). 
$$
Here each edge corresponds to one linear factor
$x_j-x_i$, so the whole $F_G$ is defined up to
a sign. Assume that each vertex $v_i$ has a list $A_i$
consisting of $d_i+1$ colors, which are real numbers.
A \emph{proper list coloring of $G$ subordinate to lists 
$\{A_i\}_{1\leqslant i\leqslant n}$} is a choice of colors $\a=(a_1,\ldots,a_n)\in
A_1\times \ldots \times A_n=A$ for which
neighbouring vertices have different colors:
$a_i\ne a_j$ whenever $(i,j)\in E$. In other words,
a proper list coloring is a choice of $\a\in A$
for which $F_G(\a)\ne 0$. 
If $|\d|=|E|$, the existence of a proper list
coloring follows from $[\x^\d]F_G\ne 0$.

Define the \emph{chromatic number} $\chi(G)$ of the graph $G$ as the minimal $m$ such that there exists a proper list coloring of $G$ subordinate to equal lists of size $m$: $A_i=\{1,\dots,m\}$. Define the \emph{list chromatic number} 
$\ch(G)$ of the graph $G$
as the minimal
$m$ such that for arbitrary lists $A_i$,
$|A_i|\geqslant m$, there exists a proper
list coloring of $G$ subordinate to these lists.
Define the \emph{Alon--Tarsi number} $AT(G)$ of the graph 
$G$ as the minimal $k$ for which there exists 
a monomial $\x^\d$ such that $\max(d_1,\ldots,d_n)=k-1$
and $[\x^\d]F_G\ne 0$.

From above we see that the list chromatic number
does not exceed the Alon--Tarsi number: 
\begin{equation}\label{choice-and-AT}
 \ch(G)\leqslant \AT(G).   
\end{equation}

Further we consider the Alon--Tarsi
numbers for the graphs which are direct products
$G_1\square G_2$ of simpler graphs $G_1=(V_1,E_1)$ and $G_2=(V_2,E_2)$.
Recall that the vertex set of 
$G_1\square G_2$ is $V_1\times V_2$
and two pairs $(v_1,v_2)$ and $(u_1,u_2)$
are joined by an edge if and only if either
$v_1=u_1$ and $(v_2,u_2)\in E_2$ or
$v_2=u_2$ and $(v_1,u_1)\in E_1$.

It is well known (Lemma 2.6 in
\cite{Sabidussi1957}) 
that $\chi(G_1\square G_2)=\max(\chi(G_1),\chi(G_2))$. Much less is known about the list chromatic number (and the Alon--Tarsi number) of the Cartesian product of graphs. Borowiecki, Jendrol, Kr{\'{a}}l, and Mi{\v{s}}kuf \cite{Borowiecki2006} gave the following bound:

\begin{theorem}[\cite{Borowiecki2006}]\label{thm-colbound}
For any two graphs $G$ and $H$,
\[
\ch(G\square H)\leqslant\min(\ch(G)+\col(H),\col(G)+\ch(H))-1.
\]
\end{theorem}

Here $\col(G)$ is the \emph{coloring number} of $G$, i.e. the smallest integer $k$ for which there exists an ordering of vertices $v_1,\dots,v_n$ of $G$ such that each vertex $v_i$ is adjacent to at most $k-1$ vertices among $v_1,\dots,v_{i-1}$.

Here we continue the previous work
\cite{li2019alontarsi} where the toroidal
grid $C_n\square C_m$ (here $C_n$ is a 
simple cycle with $n$ edges) was considered and it was proved that
$\AT(C_{n}\square C_{2k})=3$.

An explicit form of Combinatorial Nullstellensatz (the coefficient formula) was used in \cite{li2019alontarsi}, such approach does not seem to work
in the more general setting of the present paper.

We call a coefficient
$\left[\x^\xi\right]F_G(\x)$ of 
the graph polynomial $F_G$ \emph{central},
if $\xi_i=\deg_G(v_i)/2$ for all $i$, 
and \emph{almost central}, if  $|\xi_i-\deg_G(v_i)/2|\leqslant 1$ for all $i$.

Our main result is the following

\begin{theorem}\label{thm-main}
Let $G$ be a graph, all vertices in which have even degree. Suppose that the graph polynomial $F_G$ has at least one non-zero almost central coefficient. Then for $H=G\square C_{2k}$ the central coefficient is non-zero.
In particular,
$H$ is $(\deg_H/2+1)$-choosable and
\[
\ch(H)\leqslant \AT(H)\leqslant 
\frac{\Delta(H)}{2}+1=
\frac{\Delta(G)}{2}+2.
\]
\end{theorem}

Note that Theorem~\ref{thm-colbound} gives the bound $\ch(H)\leqslant\min(\ch(G)+2,\col(G)+1)$ under the same conditions. 
When $\ch(G)$ (or $\col(G)$) is small,
this bound is stronger. 
But it can also be weaker when $\ch(G)$ and $\col(G)$ are close to $\Delta(G)$. For example, if $G=C_{2l+1}$ is an odd cycle, then $F_G$ obviously has a non-zero almost central coefficient, so, by Theorem~\ref{thm-main}, $\ch(C_{2l+1}\square C_{2k})\leqslant 3$
(this was also proved in \cite{li2019alontarsi} by a different argument). On the other hand, Theorem~\ref{thm-colbound} gives only $\ch(C_{2l+1}\square C_{2k})\leqslant 4$.

\section{Coefficients as traces}

Let $\a=(a_1,\ldots,a_n), \b=(b_1,\ldots,b_n)\in \mathbb{Z}_{\geqslant 0}^n$; denote $a_1+\dots+a_n=|\a|$. Consider a polynomial
\[
P(\x,\y) = Q(\x)R(\x,\y)\in \mathbb{F}[\x,\y]
\]
in variables $\x=(x_1,\ldots,x_n)$,
$\y=(y_1,\ldots,y_n)$, where $Q$ is of degree at most $|\a|$, $R$ is homogeneous of degree $|\b|$.

Consider $nk$ variables $(\x^1,\x^2,\ldots,\x^k)$,
$\x^i=(x^i_1,\ldots,x^i_n)$; for convenience, denote
$\x^{k+1}\equiv \x^1$.
Define
$$
P_k(\x^1,\x^2,\ldots,\x^k)=
\prod_{1\leqslant j\leqslant k} P(\x^j,\x^{j+1}).
$$
We are interested in the coefficient
$$
M_k:=\left[\prod_{j=1}^k(\x^j)^{\a+\b}\right] P_k.
$$

It is easy to see that this coefficient is equal to
\[
\sum \prod_{j=1}^k \left[(\x^j)^{\p^j}(\x^{j+1})^{\q^j}\right]R(\x^j,\x^{j+1}) \cdot \left[(\x^j)^{\a+\b-\p^j-\q^{j-1}}\right]Q(\x^j)= \tr \Phi^k,
\]
where the sum is over all $(\p^1,\dots,\p^k)$, $(\q^1,\dots,\q^k)$ such that
\[
\p^j=(p^j_1,\dots,p^j_n),\q^j=(q^j_1,\ldots,q^j_n)\in \mathbb{Z}_{\geqslant 0}^n,\quad |\p^j|+|\q^j|=|\b|;
\]
for $\alpha=(\alpha^1,\alpha^2)$ and $\beta=(\beta^1,\beta^2)$; $\alpha^i,\beta^i\in \mathbb{Z}_{\geqslant 0}^n$; $|\alpha^1|+|\alpha^2|=|\beta^1|+\beta^2|=|\b|$,
\[
\Phi(\alpha,\beta) = \left[(\x)^{\alpha^1}(\y)^{\alpha^2}\right]R(\x,\y) \cdot \left[(\x)^{\a+\b-\alpha^1-\beta^2}\right]Q(\x).
\]

If $Q$ is homogeneous, then $\Phi(\alpha,\beta)\neq 0$ only if $|\alpha^1|+|\beta^2|=|\b|$, i.e. if $|\alpha^1|=|\beta^1|$, $|\alpha^2|=|\beta^2|$; the same is true for $\Phi(\beta,\alpha)$.

\section{Cartesian product of a graph and an even cycle}

Let $G$ be a graph, all vertices in which have even degree. Denoting $V(G)=\{v_1,\dots,v_n\}$, we take $a_i=\deg(v_i)/2$, $Q(\x) = F_G(\x)$, $R(\x,\y) = \prod_j (y_j - x_j)$, $b_i = 1$, and let $k$ be even.

Then $\Phi(\alpha, \beta) \neq 0$ only if $\alpha^1_j=1-\alpha^2_j \leq 1,\beta^1_j=1-\beta^2_j\leq 1$ for all $j$, $|\alpha^1|=|\beta^1|$; if this is the case, then
\[
\Phi(\alpha,\beta) = (-1)^{|\alpha^1|}\left[(\x)^{\a+\b-\alpha^1-(\b-\beta^1)}\right]F_G(\x)=(-1)^{|\alpha^1|}\left[(\x)^{\a+\beta^1-\alpha^1}\right]F_G(\x).
\]
Note that
\begin{equation}\label{symmetry}
    \left[\x^{\a+\beta^1-\alpha^1}\right]F_G(\x)=
(-1)^{|E(G)|} \left[\x^{\a+\alpha^1-\beta^1}\right]F_G(\x),
\end{equation}
since simultaneously changing the choice in each linear factor $x_i-x_j$ of $F_G$ we get one of these monomials from the other. Thus, matrix $\Phi$ is (skew-)symmetric; therefore
all eigenvalues of the matrix are real (or all are imaginary). Then the $k$-th powers of all non-zero eigenvalues are real and have the same sign. It follows that $\tr\Phi^k\neq 0$ if at least one of the coefficients of the form
\[
\left[\x^{\a+\beta^1-\alpha^1}\right]F_G(\x)
\]
is non-zero; in other words, if there is at least one non-zero coefficient $\left[\x^\xi\right]F_G(\x)$ with $|\xi_i-\deg(v_i)/2|\leqslant 1$ for all $i$. Theorem \ref{thm-main} is proved.

\begin{remark}\label{remark}
Define a generalized graph polynomial 
$Q$ for a graph or multigraph $G=(V,E)$
with even degrees, 
$V=\{v_1,\ldots,v_n\}$,
as a product of factors $x_i\pm x_j$
for all edges $v_iv_j\in E$ (one multiple
for each edge). Note that it satisfies the symmetry or antisymmetry property \eqref{symmetry}, with some multiple 
$\pm 1$ on the place
of $(-1)^{|E(G)|}$. Therefore the same
argument shows that if $Q$ has
a non-zero almost central coefficient, then the polynomial
\begin{equation}\label{gengr}
\prod_{i=1}^{2k}
Q(\x^i,\x^{i+1})\prod_{i=1}^{2k} \prod_{j=1}^n (x^i_j-x^{i+1}_j),\,\, \text{where}\,\,\x^{2k+1}\equiv \x^1,
\end{equation}
has a non-zero central coefficient
(that is, a coefficient of $\prod_{i,j}
(x^i_j)^{\deg(v_j)/2+1}$).
\end{remark}

\section{Applications}

\subsection{Cartesian product of several cycles}

Consider a Cartesian product of odd cycles $G=C_{2k_1+1}\square\dots\square C_{2k_n+1}$, such that
\[
\frac{1}{k_1}+\dots+\frac{1}{k_n}\leqslant 1.
\]

Our goal is to show that the graph polynomial $F_G$ has a non-zero almost central coefficient. We employ the Alon-Tarsi method:

\begin{theorem}[see Corollary 1.2, Corollary 2.3 in \cite{Alon1992}]
Let $G$ be a non-directed graph on vertices $v_1,\dots,v_n$. Suppose $G$ has an orientation $D$ with outdegrees $d_{out}(v_i)=d_i$, and there are no odd directed cycles in $D$. Then the coefficient $[\x^\d]F_G$ (where $\d=(d_1,\dots,d_n)$) is non-zero.
\end{theorem}

We are going to build an orientation of $G$, such that the outdegree of any vertex lies in $\{n-1,n,n+1\}$, and there are no odd directed cycles. We are going to denote vertices of $G$ by $v=(v_1,\dots,v_n)$, $0\leqslant v_i\leqslant 2k_i$.

We divide $G$ into $2^n$ boxes $H_i$, $0\leqslant i<2^n$: if the binary notation of $i$ is $\overline{b_{i,1}\dots b_{i,n}}$, then
\[
H_i=\{(v_1,\dots,v_n)\ |\ 0\leqslant v_j\leqslant k,\text{ if } b_{i,j}=0;\ k<v_j\leqslant 2k \text{ otherwise}\}.
\]

These boxes may be 
colored alternately white and black (in a chess order).
We direct all edges sticking out of 
black boxes outward, and 
from white boxes inward.
 Note that any directed cycle is contained is some 
 box $H_i$ and has therefore even length.

The remaining task is to obtain the orientation of the box of dimension $n$ with outdegree of any vertex lying in $\{n-1,n\}$. We will use this orientation for all white boxes $H_i$'s, for
the
black boxes use the reversed orientation. This guarantees that
in white boxes all outdegrees are in $\{n-1,n\}$;
in black boxes the indegrees are in $\{n-1,n\}$,
therefore the outdegrees are in $\{n,n+1\}$.

To prove the existence of such orientation we are going to use the following theorem (see, for example, Theorem 3 in \cite{Frank1976HowTO}):

\begin{theorem}\label{thm-orient}
There exists an orientation of $G=(V,E)$ with $l_v\leqslant d_{out}(v)\leqslant u_v$ for any $v$, if and only if for any $W\subset V$ the following two conditions hold:
\begin{enumerate}
\item $|E(W)|\leqslant \sum_{v\in W} u_v$;
\item $|\overline{E}(W)|\geqslant \sum_{v\in W} l_v$,
\end{enumerate}
where $\overline{E}(W)=E(V)\setminus E(V\setminus W)$ is the set of edges incident to at least one vertex of $W$.
\end{theorem}

\begin{proposition}
Let $H=P_{k_1}\square\dots\square P_{k_n}$ ($P_i$ is a path of length $i$). There exists an orientation of $H$ with outdegrees of all vertices lying in $\{n-1,n\}$ if and only if
\begin{equation}\label{cond}
\frac{1}{k_1}+\dots+\frac{1}{k_n}\leqslant 1.
\end{equation}
\end{proposition}
\begin{proof}
First of all, note that the condition~\eqref{cond} is necessary: the sum of outdegrees of all vertices does not exceed the number of edges in a graph, so
\begin{equation}\label{sumdegs}
(n-1)\prod_{i=1}^n k_i\leqslant \sum_{i=1}^n (k_i-1)\prod_{\substack{1\leqslant j\leqslant n\\ j\neq i}}k_i=\prod_{i=1}^n k_i\cdot \sum_{i=1}^n\left(1-\frac{1}{k_i}\right),
\end{equation}
which is equivalent to \eqref{cond}.
To show that the condition~\eqref{cond} is sufficient, we are going to verify two conditions from Theorem~\ref{thm-orient} with $l(v)=n-1$, $u(v)=n$ for each $v$.
The first condition holds:
\[
|E(W)|\leqslant \frac{1}{2}\sum_{v\in W} d(v)\leqslant n|W|.
\]

The second condition looks like
\[
|E(V)|-|E(V\setminus W)|\geqslant (n-1)|W|.
\]
Denoting $U=V\setminus W$, it is equivalent to
\[
|E(U)|-(n-1)|U|\leqslant |E(V)|-(n-1)|V|
\]
for each $U\subset V$. Thus, to prove that the second condition holds, it is sufficient to show that the function
\[
f(U)=|E(U)|-(n-1)|U|
\]
reaches its maximum value at $U=V$.

For $1\leqslant i\leqslant n$, $p=(p_1,\dots,p_{i-1},p_{i+1},\dots,p_n)$, $1\leqslant p_j\leqslant k_j$, denote
\[
U(i,p)=\{v\in U\ |\ v_j=p_j \text{ for any } j\neq i\}.
\]
Then
\[
|E(U)|\leqslant n|U|-\sum_{i,p}\chi(|U(i,p)|>0).
\]
Denote $l=1-\sum_{i=1}^n\frac{1}{k_i}\geqslant 0$; then
\[
|U|=\left(\sum_{i=1}^n\frac{1}{k_i}+l\right)|U|=l|U|+\sum_{i,p}\frac{|U(i,p)|}{k_i}.
\]
It follows that
\[
f(U)\leqslant l|U|+\sum_{i,p} g(U,i,p),
\]
where
\[
g(U,i,p)=\begin{cases}
0,          & \text{ if } |U(i,p)|=0,\\
\frac{|U(i,p)|}{k_i}-1, & \text{ otherwise.}
\end{cases}
\]
In conclusion, note that
\[
f(U)\leqslant l|U|+\sum_{i,p}g(U,i,p)\leqslant l|U|\leqslant l|V|=f(V).
\]
\end{proof}

\begin{corollary}
Let $G=C_{2k_1+1}\square\dots\square C_{2k_m+1}\square C_{2k_{m+1}}\square\dots\square C_{2k_n}$, $0\leqslant m<n$,
and, additionally,
\[
\frac{1}{k_1}+\dots+\frac{1}{k_m}\leqslant 1.
\]
Then
\[
\ch(G)\leqslant \AT(G)=n+1.
\]
\end{corollary}
\begin{proof}
The upper bound follows from the construction described above; the lower bound is obvious: in each monomial of the graph polynomial $F_G$ there is a variable of degree at least $n$.
\end{proof}

Note that even though the result is sharp for the Alon--Tarsi number, it is far from sharp for the list chromatic number
when $n$ is large enough. For example, if $k_i\geqslant 2$ for all 
$i=1,\ldots,m$
(it is so for sure if $m\geqslant 2$ and 
$\sum 1/k_i\leqslant 1$), then graph $G$ is triangle free and has maximum degree $2n$, which yields $\ch(G)\leqslant (2+o(1))\frac n{\log n}$ by a result of Molloy \cite{Molloy2019}, a recent improvement of the
$O(\frac{n}{\log n})$ bound first given by Johansson \cite{johansson}.

\subsection{Powers of cycles}

\begin{proposition}
Let $C_n^p$ be $p$-th power of a cycle $C_n$, i.e. a graph on the vertex set $\{v_1,\dots,v_n\}$, in which $v_i$ and $v_j$ are adjacent if and only if $j\in\{i-p,\dots,i-1,i+1,\dots i+p\}$ (the indices are modulo $n$). Suppose $p+1$ divides $n$ or $n\geqslant p(p+1)$. Then
\[
\ch(C_n^p\square C_{2k})\leqslant \AT(C_n^p\square C_{2k})\leqslant p+2.
\]
\end{proposition}
\begin{proof}
In \cite{prowsewoodall} the Alon--Tarsi number 
$\AT$ for powers of cycles is estimated. If $p+1$ divides $n$, it is shown that the central coefficient of $F_{C_n^p}$ is non-zero; if $p+1$ does not divide $n$, but $n\geqslant p(p+1)$, then it is shown that for a graph $H_n^p$, obtained by adding some matching to the graph $C_n^p$, there is a non-zero coefficient of $F_{H_n^p}$ with degree of each variable in $\{p,p+1\}$. This coefficient is a linear combination of almost central coefficients of $F_{C_n^p}$, so at least one of them is also non-zero.
\end{proof}

\subsection{Multigraphs}

Note that Theorem~\ref{thm-main} can be applied to multigraphs. In particular, non-trivial bounds can be obtained for graphs with large choice number by adding multiple edges to them. To give an example, we prove the following proposition:

\begin{proposition}\label{6}
Let $G$ be a graph, all vertices of maximum degree in which may be covered by some vertex-disjoint cycles. Then
\[
\AT(G\square C_{2k})\leqslant\Delta(G)+1=\Delta(G\square C_{2k})-1.
\]
\end{proposition}
\begin{proof}
Denote the set of edges contained in these cycles as $F$. Consider a graph $G'$, which can be obtained from $G$ by adding a multiple edge to every edge from the set $E(G)\setminus F$ . Obviously, $\AT(G\square C_{2k})\leqslant \AT(G'\square C_{2k})$. If we show that $F_{G'}$ has a non-zero almost central coefficient, then
\[
\AT(G'\square C_{2k})\leqslant \frac{\Delta(G')}{2}+2=\Delta(G)+1.
\]
Consider another graph $G''$, which can be obtained from $G$ by adding a multiple edge to every edge. Note that the central coefficient of $F_{G''}=F^2_G$ is non-zero: the central coefficient of $F_{G''}$ is the sum of products of ``opposite'' coefficients of $F_G$, each summand in this sum has the same sign (which depends on the parity of the number of edges in $G$, cf. \eqref{symmetry}). But the central coefficient of $F_{G''}$ is a linear combination of almost central coefficients of $F_{G'}$; it follows that at least one of them is also non-zero.
\end{proof}

\begin{corollary}\label{nine}
\[
\AT(K_n\square C_{2k})=\ch(K_n\square C_{2k})=n.
\]
\end{corollary}

\begin{proof}
We have $$n\geqslant \AT(K_n\square C_{2k})
\geqslant \ch(K_n\square C_{2k})\geqslant
\ch(K_n)\geqslant n,$$
where the first inequality follows from 
Proposition \ref{6}, the second from
\eqref{choice-and-AT}, the third and
fourth are clear. So all inequalities turn
into equalities.
\end{proof}

Next proposition is a generalization of Theorem \ref{thm-main} for arbitrary graphs (not necessarily with even degrees,
not necessarily with a non-zero almost central coefficient.)
Roughly speaking, it bounds the choosability in dependence on how not-so-far-from-central coefficient does
the graph polynomial have. It also
gives Corollary \ref{nine} (we skip the details).

\begin{proposition}
Let $G=(V,E)$ be a graph; denote $V=\{v_1,\dots,v_n\}$. For any $\eta=(\eta_1,\dots,\eta_n)$ denote $l_G(\eta,i)=|\eta_i-\deg_G(v_i)/2|$.
Consider a non-zero coefficient $[\x^\tau]F_G(\x)$ of the graph polynomial $F_G$. 
Partition the set $\{1,\ldots,n\}$
onto sets
\begin{align*}
    N&=\{i:\tau_i=\deg_G(v_i)/2\},\\
    A_1&=\{i:\tau_i\leqslant \deg_G(v_i)/2-1\},\\
    A_2&=\{i:\tau_i=\deg_G(v_i)/2-1/2\},\\
    B_1&=\{i:\tau_i\geqslant \deg_G(v_i)/2+1\},\\
    B_2&=\{i:\tau_i=\deg_G(v_i)/2+1/2\}.
\end{align*}

Additionally, let $A_3$ be an arbitrary subset of $A_1$ of size $\max(0,|A_1|-|B_1|)$; let $B_3$ be an arbitrary subset of $B_1$ of size $\max(0,|B_1|-|A_1|)$. Then $G\square C_{2k}$ is $f$-choosable, where
\[
f(v_i)=\begin{cases}
\deg_G(v_i)/2+2, & \text{ if } i\in N,\\
\deg_G(v_i)/2+l_G(\tau,i)+1, & \text{ if } i\in A_1\cup B_1\setminus A_3\setminus B_3,\\
\deg_G(v_i)/2+l_G(\tau,i)+2, & \text{ if } i\in A_2\cup B_2\cup A_3\cup B_3.
\end{cases}
\]
\end{proposition}
\begin{proof}
Define a multiset $A$ as follows:
\begin{itemize}
    \item each $i\in A_1\setminus A_3$ occurs $2(l_G(\tau,i)-1)$ times in $A$;
    %\item each $i\in A_2\cup $ occurs in $A$ once;
    \item each $i\in A_2\cup A_3$ occurs $2l_G(\tau,i)$ times in $A$.
\end{itemize}
Multiset $B$ is defined similarly. Note that
\[
|A|=\sum_{i\in A_1\cup A_2}2l_G(\tau,i)-2\min(|A_1|,|B_1|).
\]
Similarly,
\[
|B|=\sum_{i\in B_1\cup B_2}2l_G(\tau,i)-2\min(|A_1|,|B_1|).
\]
It follows that $|A|=|B|=:m$. Let $A=\{a_1,\dots,a_m\}$, $B=\{b_1,\dots,b_m\}$.

Consider $2^m$
polynomials
$$
Q_{\varepsilon}(\x)=F_G(\x)\cdot\prod_{j=1}^m(x_{a_j}\pm x_{b_j})
$$
indexed by the choice $\varepsilon$ of $m$ signs.
Using the relation $x_{a_j}=\frac12(
(x_{a_j}+x_{b_j})+(x_{a_j}-x_{b_j}))$
we see that
the polynomial $Q:=F_G(\x)\cdot \prod x_{a_i}$
is a linear combination of $Q_{\varepsilon}$'s. Note that $$\left[\x^\tau\cdot \prod x_{a_i}
\right] Q=\left[\x^\tau\right]F_G\ne 0,$$
therefore there exists $\varepsilon$
such that
$$
\left[\x^\tau\cdot \prod x_{a_i}
\right] Q_\varepsilon \ne 0.
$$

Note that $Q_{\varepsilon}$
is a generalized graph polynomial of a
certain multigraph on the ground set $V$
with degree function
$2f(v_i)-4$. The coefficient of
$\x^\tau\cdot \prod x_{a_i}$ 
is almost central for $Q_{\varepsilon}$.

Now it follows from the Remark \ref{remark} that the polynomial \eqref{gengr} (for
$Q=Q_\varepsilon$)
has a non-zero central coefficient. 

Finally, since the graph polynomial $F_{G\square C_{2k}}$ divides this polynomial, graph $G\square C_{2k}$ is $f$-choosable.
\end{proof}

We are grateful to Noga Alon for bringing the papers \cite{johansson,Molloy2019}
to our attention.

The study was funded by Foundation for the Advancement of Theoretical Physics and Mathematics ``BASIS'' (general
Sections 2 and 3) and by
by RFBR, project number 19-31-90081
(applications in Section 4).

\bibliographystyle{plain}
\bibliography{bibliography}

\begin{thebibliography}{1}

\bibitem{Alon1999comb}
N.~Alon.
\newblock Combinatorial {N}ullstellensatz.
\newblock {\em Combinatorics, Probability and Computing}, 8(1-2):7--29, 1999.

\bibitem{Alon1992}
N.~Alon and M.~Tarsi.
\newblock Colorings and orientations of graphs.
\newblock {\em Combinatorica}, 12(2):125--134, 1992.

\bibitem{Borowiecki2006}
M.~Borowiecki, S.~Jendrol, D.~Kr{\'{a}}l, and J.~Mi{\v{s}}kuf.
\newblock List coloring of {Cartesian} products of graphs.
\newblock {\em Discrete Mathematics}, 306(16):1955--1958, 2006.

\bibitem{Frank1976HowTO}
A.U. Frank and A.~Gy{\'a}rf{\'a}s.
\newblock How to orient the edges of a graph?
\newblock {\em Colloq. Math.Soc. J{\'a}nos Bolyai}, 18:353--364, 1976.

\bibitem{johansson}
A.~Johansson.
\newblock Asymptotic choice number for triangle free graphs.
\newblock {\em DIMACS Technical Report 91–5}, 1996.

\bibitem{li2019alontarsi}
Z.~Li, Z.~Shao, F.~Petrov, and A.~Gordeev.
\newblock {The Alon--Tarsi Number of A Toroidal Grid}.
\newblock {\em arXiv preprint arXiv:1912.12466}, 2019.

\bibitem{Molloy2019}
M.~Molloy.
\newblock The list chromatic number of graphs with small clique number.
\newblock {\em Journal of Combinatorial Theory, Series B}, 134:264--284, 2019.

\bibitem{prowsewoodall}
A.~Prowse and D.R. Woodall.
\newblock {Choosability of Powers of Circuits}.
\newblock {\em Graphs and Combinatorics}, 19:137--144, 2003.

\bibitem{Sabidussi1957}
G.~Sabidussi.
\newblock Graphs with given group and given graph-theoretical properties.
\newblock {\em Canadian Journal of Mathematics}, 9:515--525, 1957.

\end{thebibliography}

\end{document}